\documentclass[a4paper,10pt]{article}
\usepackage[centertags]{amsmath}
\usepackage{amsfonts}
\usepackage{amssymb}
\usepackage{amsthm}
\usepackage{epsfig}
\usepackage{setspace}
\usepackage{ae}
\usepackage{eucal}
\usepackage[usenames]{color}%
\usepackage{graphicx}

\theoremstyle{plain}
\newtheorem{thm}{Theorem}[section]
\newtheorem{lem}[thm]{Lemma}

\newtheorem{prop}[thm]{Proposition}
\theoremstyle{definition}

\theoremstyle{remark}

\newtheorem{rem}[thm]{Remark}

\title{Generalisation of Hajek's stochastic comparison results to stochastic sums}
\author{J\"org Kampen }

\begin{document}

\maketitle

\begin{abstract} 
Hajek's stochastic comparison result is generalised to multivariate stochastic sum processes with univariate convex data functions  and for univariate monoton nondecreasing convex data functions for processes with and without drift respectively. The univariate result is recovered.
\end{abstract}


2010 Mathematics Subject Classification 60H10. 
\section{Statement of results}

Applications of the extension of Hajek's results to stochastic sums were described in \cite{KAB} and \cite{KPP}, but a full proof was not given in these notes. Here we give a short complete proof of related results. Hajek's results are recovered by the different proof method. In the following $C({\mathbb R})$ denotes the functions space of continuous functions on the field of real numbers ${\mathbb R}$, $W$ denotes a standard $N$-dimensional Brownian motion, and $E^x$ denotes the expectation of a process starting at $x\in {\mathbb R}^n$. Furthermore, for a ${\mathbb R}^n$-valued process $(X(t))_{0\leq t\leq T}$  the $i$th component of this process is denoted by $X_i(t)$.
For processes without drift we prove
\begin{thm}\label{main1}
Let $T>0$, $f\in C({\mathbb R})$ be convex, and assume that $f$ satisfies an exponential growth condition. Assume that $c_i>0$ are some positive real constants for $1\leq i\leq n$. Furthermore, let $X,Y$ be semimartingales with $x=X(0)=Y(0)$, where
\begin{equation}
X(t)=X(0)+\int_0^t\sigma\left(X(s)\right)dW(s),
\end{equation}
\begin{equation}
Y(t)=Y(0)+\int_0^t\rho\left(Y(s)\right)dW(s),
\end{equation}
with $n\times n$-matrixvalued bounded Lipschitz-continuous functions $x\rightarrow \sigma\sigma^T(x)$ and $y\rightarrow \rho\rho^T$. If $\sigma\sigma^T\leq \rho\rho^T$, then for $0\leq t\leq T$ we have
\begin{equation}
E^x\left(f\left(\sum_{i=1}^nc_i X_i(t)\right) \right)\leq E^x\left(f\left(\sum_{i=1}^nc_i Y_i(t)\right) \right)
\end{equation}
Here, we say that $\sigma\sigma^T\leq \rho\rho^T$ if for all $x\in {\mathbb R}^n$ the matrix $\sigma\sigma^T(x)-\rho\rho^T(x)$ is positive.
\end{thm}

For processes with drift we prove
\begin{thm}\label{main2}
Let $T>0$, $f\in C({\mathbb R})$ be nondecreasing, convex, and assume that $f$ satisfies an exponential growth condition. Assume that $c_i>0$ are some positive real constants for $1\leq i\leq n$. Furthermore, let $X,Y$ be semimartingales with $x=X(0)=Y(0)$, where
\begin{equation}
X(t)=X(0)+\int_0^t\mu(X(s))ds+\int_0^t\sigma\left(X(s)\right)dW(s),
\end{equation}
\begin{equation}
Y(t)=Y(0)+\int_0^t\nu(Y(s))ds+\int_0^t\rho\left(Y(s)\right)dW(s),
\end{equation}
with bounded Lipshitz-continuous drift functions $\mu\leq \nu$ and $n\times n$-matrix-valued bounded Lipschitz-continuous functions $x\rightarrow \sigma\sigma^T(x)$ and $y\rightarrow \rho\rho^T$. If $\mu\leq \nu$ and $\sigma\sigma^T\leq \rho\rho^T$, then for $0\leq t\leq T$ we have
\begin{equation}
E^x\left(f\left(\sum_{i=1}^nc_iX_i(t)\right) \right)\leq E^x\left(f\left(\sum_{i=1}^nc_iY_i(t)\right) \right)
\end{equation}
Here, we say that $\mu\leq \nu$ if for all $x\in {\mathbb R}^n$ the difference $\mu(x)-\nu(x)$ is positive for each component.
\end{thm}
\begin{rem}
Bounded Lipschitz-continuity, i.e., the condition that for some $C>0$
\begin{equation}
|b(x)+|\sigma\sigma^T(x)|\leq C,~|b(x)-b(y)|+|\sigma(x)-\sigma(y)|\leq C|x-y|
\end{equation}
holds for all $x,y\in {\mathbb R}^n$ implies existence of a $t$-continuous solution in stochastic $L^2$ sense of $X$. Similarly for $Y$. Proofs are based on a generalisation of OD-proofs to infinite dimensional function spaces and can be found in elementary standard textbooks such as \cite{O}.
\end{rem}

\section{Proof of Theorem \ref{main1} }

We first remark that the intial data function has to be univariate - for a general multivariate data function $f$ the results do not hold, because simple examples show that convexity can be strongly violated in this genera situation.
Since classical representations of the value functions in terms of the probability density (fundamental solution) are not convolutions we use the adjoint of the fundamental solution. For this and other technical reasons we need some more regularity of the data function and the the diffusion matrix $\sigma\sigma^T$ in order to to treat the problem on an analytical level. We shall observe then that the pointwise result is preserved as we consider certain data and coefficient function limits reducing the regularity assumptions.
First we need some regularity assumptions which ensure that the fundamental solution and the adjoint fundamental solution existence in a classical sense, i.e. have pointwise well-defined spatial derivatives up to second order and a pointwise well defined partial time derivative up to first order (in the domain where it is continuous). For  the sake of possible generalisations in the next section we consider the more general operator
\begin{equation}
L\equiv\frac{\partial}{\partial t}-\sum_{i,j=1}^na_{ij}\frac{\partial^2 }{\partial xi\partial x_j}-\sum_{i=1}^nb_i\frac{\partial}{\partial x_i}-c.
\end{equation}
We include even the potential term coefficients $c$ because such a coefficient appears in the adjoint even f $c=0$. Recall that the adjoint operator is given by
\begin{equation}
L^*\equiv-\frac{\partial}{\partial t}+\sum_{i,j=1}^na^*_{ij}\frac{\partial^2 }{\partial x_i\partial x_j}+\sum_{i=1}^nb^*_i\frac{\partial}{\partial x_i}+c^*,
\end{equation}
where
\begin{equation}
a^*_{ij}=a_{ij},~b_i^*=2\sum_{j=1}^n a_{ij,j}-b_i,~c^*=c+\sum_{i,j=1}^na_{ij.,i,j}-\sum_{i=1}^nb_{i,i}.
\end{equation}
In this section we shallassume that $b_i\equiv 0$ and $c\equiv 0$. Note that even in this restrictive situation we have $b^*_i\neq 0$ and $c^*\neq 0$. For our purposes it suffices to assume that the coefficients are of spatial dependence (the generalisation to additional time dependence is straightforward). In order that the adjoint exists in a strong sense we should have bounded continuous derivatives.

We assume  
\begin{itemize}
\item[i)]\begin{equation}\label{c1}
\forall~1\leq i,j\leq n~~a_{ij}\in C^2\cap H^2,
\end{equation}
where $C^m\equiv C^m\left({\mathbb R}^n\right) $ denotes the space of real-valued twice continuously differentiable functions and $H^m$ denotes the standard Sobolev space of order $m\geq 0$. In the next section we assume in a addition that $b_i\in C^1\cap H^1$ for all $1\leq i\leq 1$. For the folowing considerations concerning the adjoint we assume that $c\in C^0\cap L^2$ if a potential coefficient is considered. 
\item[ii)] we have uniform ellipticity, i.e., there exists $0<\lambda <\Lambda <\infty$ such that
\begin{equation}\label{c2}
\forall x,y\in {\mathbb R}^n:~\lambda |y|^2\leq \sum_{i,j=1}^na_{ij}(x)x_ix_j \leq \Lambda |y|^2.
\end{equation}
\end{itemize}
We use one observation concerning the adjoint
\begin{lem}\label{lem1}
Assume that the conditions (\ref{c1}) and (\ref{c2}) above hold and let $p$ be the fundamental solution of
\begin{equation}
Lp=0,
\end{equation}
and $p*$ be the fundamental solution of
\begin{equation}
L^*p^*=0.
\end{equation}
Then for $s<t$ and $x,y\in {\mathbb R}^n$ $p,p^*$ have spatial derivatives up to order $2$
\begin{equation}
\begin{array}{ll}
p(t,x;s,y)=p^*(s,y;t,x),~p_{,i}(t,x;s,y)=p^*_{,i}(s,y;t,x),\\
\\
~p_{,i,j}(t,x;s,y)=p^*_{,i,j}(s,y;t,x).
\end{array}
\end{equation}
Here, for $x=(x_1,\cdots,x_n)$ and $e_i=(e_{i1},\cdots, e_{in})$ along with $e_{ij}=\delta_{ij}$ (Kronecker $\delta$) and $s<t$, $x,y\in {\mathbb R}^n$ we denote
\begin{equation}
p_{,i}(t,x;s,y)=\lim_{h\downarrow 0}\frac{p(t,x+he_i;s,y)-p(t,x;s,y)}{h},
\end{equation}
\begin{equation}
p_{,i,j}(t,x;s,y)=\lim_{h\downarrow 0}\frac{p_{,i}(t,x+he_j;s,y)-p_{,i}(t,x;s,y)}{h},
\end{equation}
\begin{equation}
p^*_{,i}(t,x;s,y)=\lim_{h\downarrow 0}\frac{p^*(s,y+he_i;t,x)-p^*(s,y;t,x)}{h},
\end{equation}
\begin{equation}
p^*_{,i,j}(t,x;s,y)=\lim_{h\downarrow 0}\frac{p^*_{,i}(t,s,y+he_j;t,x)-p^*_{,i}(s,y;t,x)}{h}.
\end{equation}
\end{lem}

\begin{proof}
For $q(\tau,z)=p(\tau,z;s,y)$ and $r(\tau,z)=p^*(\tau,z;t,x)$ for $s<\tau<t$ we show that for $1\leq i,j\leq n$
\begin{equation}
q(t,x)=r(s,y),~ q_{,i}(t,x)=r_{,i}(s,y),~q_{,i,j}(t,x)=r_{,i,j}(s,y)
\end{equation}
hold. Let $B_{R}$ be the ball of radius $R$ around zero. As $s<t$ there exists $\delta>0$ such that $s+\delta<t-\delta$ and using Green's identity, Gaussian upper bounds of the fundamental solution and its first order spatial derivatives, $Lq=0$ and $L^*r=0$ we get
\begin{equation}
\begin{array}{ll}
0=\lim_{R\uparrow \infty}\int_{s+\delta}^{t-\delta}\int_{B_{R}}
\frac{\partial }{\partial \tau}(qr)(\tau,z)d\tau dz\\
\\
=\int_{{\mathbb R}^n}q(t-\delta,z)p^*(t-\delta,z;t,x))dz-\int_{{\mathbb R}^n}r(t+\delta,z)p(t+\delta,z;s,y))dz.
\end{array}
\end{equation}
This leads to the identities
\begin{equation}
\int_{{\mathbb R}^n}q_{,i}(t-\delta,z)p^*(t-\delta,z;t,x))dz=\int_{{\mathbb R}^n}r_{,i}(t+\delta,z)p(t+\delta,z;s,y))dz,
\end{equation}
and
\begin{equation}
\int_{{\mathbb R}^n}q_{,i,j}(t-\delta,z)p^*(t-\delta,z;t,x))dz=\int_{{\mathbb R}^n}r_{,i,j}(t+\delta,z)p(t+\delta,z;s,y))dz,
\end{equation}
In the limit $\delta \downarrow 0$ we get the relations stated. 
\end{proof}

For technical reasons we need more approximations concerning the data. As we are aiming at a pointwise comparison result, and we  have Gaussian upper bounds  it suffices to consider approximating data which are regular convex in a core region and decay to zero at spatial infinity. We have

\begin{prop}\label{prop1}
Let $f\in C({\mathbb R})$ be a real-valued continuous convex function. Let $B_R\subset {\mathbb R}^n$ be the ball of finite radius $R$ around the origin. Then there is a function $f^{\epsilon}_R\in C^2\cap H^2$ such that
\begin{itemize}
 \item[i)] \begin{equation}\label{d1}
            \forall x\in B_R~f(x)=|f(x)-f^{\epsilon}_R(x)|\leq \epsilon;
           \end{equation}
\item[ii)] the second (classically well-defined) derivative is strictly positive,i.e.,
\begin{equation}\label{d2}
\forall x\in B_R~f{''}(x)>0. 
\end{equation}
\end{itemize}

\end{prop}
Proposition \ref{prop1} can be proved by using regular polynomial interpolation as considered in \cite{KN} (for example). Here it can be used that classical derivatives of second order exist for the convex continuous function $f$ almost everywhere. 
The function $f^{\epsilon, R}$ is not convex in general of course, but it is convex in a core region $B_R(x)$.
For all $\epsilon >0$ and all $R>0$ using \ref{lem1} we get
\begin{equation}\label{ver}
\begin{array}{ll}
v^{\epsilon,R}_{,i,j}(t,x)=\int_{{\mathbb R}^n}f^{\epsilon}_R\left(\sum_{i=1}^nc_iy_i\right)p_{,i,j}(t,x;0,y)dy\\
\\
= \int_{{\mathbb R}^n}f^{\epsilon}_R\left(\sum_{i=1}^nc_iy_i\right)p^*_{,i,j}(0,y;t,x)dy\\
\\
= \int_{{\mathbb R}^n}c_ic_j\left( f^{\epsilon}_R\right) ^{''}\left(\sum_{i=1}^nc_iy_i\right)p^*(0,y;t,x)dy
\end{array}
\end{equation}
Here for a univariate function $g\in C^2$ the symbol $g^{''}$ denotes its second derivative.
Since $f^{\epsilon}(z)>0$ for all $z\in B_R(z)$, and $p^*\geq 0$, and by the standard Gaussian estimate
\begin{equation}
{\big |}p^*(\sigma,\eta;\tau,\xi){\big |}\leq \frac{C^*}{\sqrt{\tau-\sigma}^n}\exp\left( -\lambda^*\frac{|\eta-\xi|^2}{\tau-\sigma}\right) 
\end{equation}
for some finite constants $C^*,\lambda^*$ we get from (\ref{ver})
\begin{equation}
\forall r>0~\forall x\in B_r~\exists R_0>r~\forall ~R\geq R_0~  \left( v^{\epsilon,R}_{,i,j}(t,x)\right)\geq 0,
\end{equation}
i.e. the Hessian is positive in a smaller core region $B_r=\left\lbrace x||x|\leq r \right\rbrace$ for $R$ large enough. Furthermore, classical regularity theory tells us that
\begin{equation}
\forall \epsilon >0~\forall R>0~v^{\epsilon,R}(t,.)\in C^{2},~v^{\epsilon}(t,.)\in C^{2},
\end{equation}
where $v^{\epsilon}(t,.)=\lim_{R\uparrow \infty}v^{\epsilon,R}(t,.)$.
It follows that 
\begin{equation}
\forall~x\in {\mathbb R}^n~\forall~t\in [0,T]~\exists R_0>0~\forall R\geq R_0~\forall~\epsilon>0~~ \mbox{Tr}A(x)D^2v^{\epsilon, R}(t,x)\geq 0,
\end{equation}
 where $A(x)=(a_{ij}(x))$ is the coefficient matrix and $D^2v^{\epsilon R}(t,x)$ is the Hessian of $v^{\epsilon R}$ evaluated at $x$. Hence,
\begin{equation}
\forall~x\in {\mathbb R}^n~\forall~t\in [0,T]~~\forall~\epsilon>0~~ \mbox{Tr}A(x)D^2v^{\epsilon}(t,x)\geq 0,
\end{equation} 
and as the $\epsilon\downarrow 0$ limit of the Hessian is well defined for $t\in (0,T]$ we get
\begin{equation}
\forall~x\in {\mathbb R}^n~\forall~t\in (0,T]~~\forall~\epsilon>0~~ \mbox{Tr}A(x)D^2v(t,x)\geq 0.
\end{equation}
Now consider matrices $(a^{v_1}_{ij})$ and $(a^{v_2}_{ij})$ where $v^1$ and $v^2$ solve
\begin{equation}
\frac{\partial v^1}{\partial t}-\sum_{ij}a^{v^1}_{ij}\frac{\partial v^1}{\partial x_i\partial x_j}=0,~
\frac{\partial v^2}{\partial t}-\sum_{ij}a^{v^2}_{ij}\frac{\partial v^2}{\partial x_i\partial x_j}=0,
\end{equation}
and $v^1(0,.)=v^2(0,.)$.
Note that $\delta v=v^1-v^2$ satisfies
\begin{equation}
\frac{\partial \delta v(t,x)}{\partial t}=\sum_{ij}\left(a^{v^2}_{ij}-a^{v^1}_{ij}\right)\frac{\partial v^1}{\partial x_i\partial x_j}+\sum_{ij}a^{v^2}_{ij}\frac{\partial^2 \delta v}{\partial x_i\partial x_j},
\end{equation}
where $\delta v(0,x)=0$ for all $x\in {\mathbb R}^n$.
We have the classical representation
\begin{equation}
\delta v(t,x)=\int_0^t\int_{{\mathbb R}^n}\sum_{ij}\left(a^{v^2}_{ij}-a^{v^2}_{ij}\right)(s,y)\frac{\partial^2 v^1}{\partial x_i\partial x_j}(s,y)p^{v^2}(t,x,s,y)dsdy,
\end{equation}
where $p^{v^2}$ is the fundamental solution of
\begin{equation}
\frac{\partial \delta v(t,x)}{\partial t}-\sum_{ij}a^{v^2}_{ij}\frac{\partial^2 \delta v}{\partial x_i\partial x_j}=0.
\end{equation}
As $\sum_{ij}\left(a^{v^2}_{ij}-a^{v^1}_{ij}\right)(s,y)\frac{\partial v^1}{\partial x_i\partial x_j}(s,y)\geq 0$ and $p^{v_2}(t,x,s,y)\geq 0$ we conclude that $\delta v^2\geq 0$.
Now we have proved the main theorem for $a_{ij}\in C^2\cap H^2$. Next, for each $\epsilon >0$ and $R>0$ there exists a matrix $(a^{\epsilon,R}_{ij})$ with components in $C^2\cap H^2$, where for all $x\in {\mathbb R}^n$
\begin{equation}
a^{\epsilon,R}(x)=\sigma^{\epsilon,R}\sigma^{\epsilon,R,T}(x),
\end{equation}
with $\sigma^{\epsilon,R,T}(x)$ the transpose of $\sigma^{\epsilon,R}$, and where
\begin{equation}
\sup_{x\in B_R}{\big |}\sigma^{\epsilon,R}(x)-\sigma(x){\big |}\leq \epsilon.
\end{equation}
here $\sigma$ is the original dispersion matrix related to the process $X$ of the main theorem (which is assumed to be bounded and Lipschitz continuous). 
\begin{equation}
X^{\epsilon, R}(t)=X(0)+
\int_0^t\sigma^{\epsilon,R}
\left(X^{\epsilon,R}(s)\right)dW(s),
\end{equation}
For a $\rho^{\epsilon,R}(x)$ which satisfies analogous conditions we define
\begin{equation}
Y^{\epsilon,R}(t)=Y(0)+\int_0^t\rho^{\epsilon,R}\left(Y^{\epsilon,R}(s)\right)dW(s),
\end{equation}
Then the preceding argument together with Feynman-Kac formalism shows that for $\sigma^{\epsilon,T}\sigma^{\epsilon,R,T}\leq \rho^{\epsilon,R}\rho^{\epsilon,R,T}$, then for $0\leq t\leq T$ we have
\begin{equation}
\begin{array}{ll}
\forall \epsilon >0~\forall r>0~\forall x\in B_r~\exists R_0~\forall R\geq R_0\\
\\
~E^x\left(f^{\epsilon,R}\left(\sum_{i=1}^nc_iX^{\epsilon,R}_i(t)\right) \right)\leq E^x\left(f^{\epsilon,R}\left(\sum_{i=1}^nc_iY^{\epsilon,R}_i(t)\right) \right)
\end{array}
\end{equation}
This leads to
\begin{equation}\label{xr}
\begin{array}{ll}
\forall r>0~\forall x\in B_r~\exists R_0~\forall R\geq R_0\\
\\
~E^x\left(f^{R}\left(\sum_{i=1}^nc_iX^{R}_i(t)\right) \right)\leq E^x\left(f^{R}\left(\sum_{i=1}^nc_iY^{R}_i(t)\right) \right)
\end{array}
\end{equation}
where $X^R$ are processes 
\begin{equation}
X^{\epsilon, R}(t)=X(0)+
\int_0^t\sigma^{R}
\left(X^{R}(s)\right)dW(s),
\end{equation}
with a bounded continuous $\sigma^R$ which satisfies
\begin{equation}
\forall~x\in B_R~\sigma^{R}(x)=\sigma(x)=0.
\end{equation}
The process $Y^R$ is defined analogously. Similarly $f^R$ is a limit of functions $f^{\epsilon,R}\in C^2\cap H^2$ which equals $f$ on $B_R$. 
In \ref{xr} $X^R$ can be replaced by $X$ and $Y^R$ by $Y$ by the probability law of the processes, and a limit consideration for data which equal for each $R$ the function $f$ on $B_R$ leads to the statement of the theorem by an uniform exponential bound of the data functions, the boundedness of the Lipshitz-continuous coefficients and the Gaussian law of the Brownian motion.

\section{Additional note for the proof of Theorem \ref{main2}}

If $w^1$ and $w^2$ solve
\begin{equation}
\begin{array}{ll}
\frac{\partial w^1}{\partial t}-\sum_{ij}a^{w^1}_{ij}\frac{\partial w^1}{\partial x_i\partial x_j}+\sum_i b^{w^1}_i(x)\frac{\partial w^1}{\partial x_i}=0,~\\
\\
\frac{\partial w^2}{\partial t}-\sum_{ij}a^{w^2}_{ij}\frac{\partial w^2}{\partial x_i\partial x_j}+\sum_ib^{w^2}_i(x)\frac{\partial w^2}{\partial x_i}=0,
\end{array}
\end{equation}
and $w^1(0,.)=w^2(0,.)$.
Note that $\delta w=w^1-w^2$ satisfies
\begin{equation}
\begin{array}{ll}
\frac{\partial \delta w(t,x)}{\partial t}=\sum_{ij}\left(a^{w^2}_{ij}-a^{w^1}_{ij}\right)\frac{\partial w^1}{\partial x_i\partial x_j}+\sum_{ij}a^{w^2}_{ij}\frac{\partial \delta w}{\partial x_i\partial x_j}-b^{w^2}_i\frac{\partial \delta w}{\partial x_i}\\
\\
+\sum_i (b^{w^2}_i(x)-b^{w^1}_i(x))\frac{\partial w^1}{\partial x_i},
\end{array}
\end{equation}
where $\delta w(0,x)=0$ for all $x\in {\mathbb R}^n$.
\begin{equation}
\begin{array}{ll}
\delta w(t,x)=\int_0^t\int_{{\mathbb R}^n}{\Big (}
\sum_{ij}\left(a^{w^2}_{ij}-a^{w^1}_{ij}\right)(s,y) w^1_{,i,j}(s,y)+\\
\\
+\sum_i(b^{w^2}_i(x)-b^{w^1}_i(x))\frac{\partial w^1}{\partial x_i}{\Big )}
p^{w^2}(t,x,s,y)dsdy,
\end{array}
\end{equation}
where $p^{w^2}$ is the fundamental solution of
\begin{equation}
\frac{\partial \delta w}{\partial t}-\sum_{ij}a^{w^2}_{ij}\frac{\partial \delta w}{\partial x_i\partial x_j}+b^{w^2}_i\frac{\partial \delta w}{\partial x_i}=0
\end{equation}

As $\sum_{ij}\left(a^{w^2}_{ij}-a^{w^1}_{ij}\right)(s,y)\frac{\partial w^1}{\partial x_i\partial x_j}(s,y)\geq 0$ and $p^{w_2}(t,x,s,y)\geq 0$ we conclude that $\delta v^2\geq 0$ if $\sum_i(b^{w^2}_i(x)-b^{w^1}_i(x))\frac{\partial w^1}{\partial x_i}{\Big )}$. As $b^{w^2}_i(x)-b^{w^1}_i(x)\geq 0$ for all $x$ this condition reduces to the monotonicity condition $\frac{\partial w^1}{\partial x_i}\geq 0$. The truth of the latter monotonicity condition for the value function $w^1$ can be proved using the adjoint using the same trick as in the preceding section. 

\begin{rem}
These notes are from my Lecture notes \begin{center}'Die Fundamentall\"{o}sung parabolischer Gleichungen und schwache Schemata h\"{o}herer Ordnung f\"{u}r stochastische Diffusionsprozesse'
\end{center} 
of WS 2005/2006 in Heidelberg, which are not published. The argument given there is published now upon request, as research is going on concerning applications of comparison principles. Originally the relevance of stochastic comparison results was pointed out to me by P. Laurence and V. Henderson.
The article \cite{KN} was never submitted to a journal and is only available on arXiv. The main theorems proved here are stated essentially in the conference notes in \cite{KAB} and \cite{KPP} but not strictly proved there. In these notes applications to American options and to passport options are considered. For example explicit solutions for optimal strategies related to the optimal control problem of passport options and the dependence of that strategy on correlations between assets can be obtained. The proof given here can be applied in the univariate case as well and recovers the result of Hajek in \cite{H}.       
\end{rem}


\begin{thebibliography}{19}
\baselineskip=12pt

\bibitem{H}
{\sc  Hayek, J.}, {\em Mean stochastic comparison results of diffusions}, Z. f. Wahrscheinichkeitstheorie u. verw. Gebiete, vol. 68.  p. 315-329, 1985.



\bibitem{KAB}
{\sc Kampen, J.}, {\em The value of the American Call with dividends increases with the basket volatility},  p. 260-272. (ISNM) Vol. 154. Birkhaeuser, 2006. (ISBN 3-7643-7718-6).


\bibitem{KPP}
{\sc Kampen, J.,} {\em On Optimal Strategies of Multivariate Passport Options}, Springer 2008, NUMB 12, pages 643-649, ISSN 612-3956.

\bibitem{KN}
{\sc Kampen, J.,} {\em Regular polynomial interpolation and approixmation of gobal solutions of linear partial differential equations}, aXiv 0807.1497v1, 2008.

\bibitem{O}
{\sc Oksendal, B.,} {\em Stochastic differential equations}, Springer 1995 (4th ed.). 


 \end{thebibliography}
\end{document}